\documentclass[ECP]{ejpecp} % replace ECP by EJP if needed.  

\usepackage{amsmath,amsthm,amssymb,tikz,mathrsfs,pgfplots}
\usepackage[notref,notcite]{}
\usepackage[numbers,comma,square,sort&compress]{natbib}
\usepackage{enumitem}
\RequirePackage{color}
\usepackage{scalerel}    %%%%  used first for scaling \bullet  
\usepackage{stmaryrd}   %%%% for integer interval notation 
\usepackage{mathabx}   %% made \widecheck work  
\usepackage{hyphenat}
\usepackage{scalerel}    %%%%  used first for scaling \bullet  

\SHORTTITLE{Upper tail of Brownian directed percolation} 

\TITLE{Upper tail large deviations in Brownian directed percolation}

\AUTHORS{%
  Christopher Janjigian\footnotemark[1]}
 % Separate items with ;

\AMSSUBJ{60F10, 60K35} % Edit. Separate items with ;
\SUBMITTED{November 3, 2018} % Edit.
\ACCEPTED{June 19, 2019} % Edit.

\VOLUME{0}
\YEAR{2012}
\PAPERNUM{0}
\DOI{vVOL-PID}

\ABSTRACT{This paper presents a new, short proof of the computation of the upper tail large deviation rate function for the Brownian directed percolation model. Through a distributional equivalence between the last passage time in this model and the largest eigenvalue in a random matrix drawn from the Gaussian Unitary Ensemble, this provides a new proof of a previously known result. The method leads to associated results for the stationary Brownian directed percolation model which have not been observed before.}

\definecolor{my-link}{rgb}{0.5,0.0,0.0}
\definecolor{my-blue}{rgb}{0.0,0.0,0.6}
\definecolor{my-red}{rgb}{0.5,0.0,0.0}
\definecolor{my-green}{rgb}{0.0,0.5,0.0}
\definecolor{nicos-red}{rgb}{0.75,0.0,0.0}
\definecolor{really-light-gray}{gray}{0.8}
\definecolor{darkgreen}{rgb}{0.0,0.5,0.0}
\definecolor{darkblue}{rgb}{0.0,0.0,0.3}
\definecolor{nicosred}{rgb}{0.65,0.1,0.1}
\definecolor{light-gray}{gray}{0.7}

\usepackage{xargs}
\usepackage[colorinlistoftodos,prependcaption,textsize=footnotesize]{todonotes}
\newcommandx{\addmath}[2][1=]{\todo[linecolor=red,backgroundcolor=red!25,bordercolor=red,#1]{#2}}
\newcommandx{\fixtext}[2][1=]{\todo[linecolor=blue,backgroundcolor=blue!25,bordercolor=blue,#1]{#2}}
\newcommandx{\note}[2][1=]{\todo[linecolor=yellow,backgroundcolor=yellow!25,bordercolor=yellow,#1]{#2}}

\allowdisplaybreaks[3]
\numberwithin{figure}{section}
\numberwithin{equation}{section}

\definecolor{sussexg}{rgb}{0,0.7,0.7}
\definecolor{sussexp}{rgb}{0.4,0,0.4}
\definecolor{sussexb}{rgb}{0.4,0.4,0.7}
\definecolor{mygray}{rgb}{0.75,0.75,0.75}

\newcommand{\E}{\mathbb{E}}

\DeclareMathOperator{\bbE}{\mathbb{E}}

\DeclareMathOperator{\bbN}{\mathbb{N}}

\DeclareMathOperator{\bbP}{\mathbb{P}}

\DeclareMathOperator{\bbR}{\mathbb{R}}

\DeclareMathOperator{\bbZ}{\mathbb{Z}}

\DeclareMathOperator{\lf}{\lfloor}
\DeclareMathOperator{\rf}{\rfloor}

\makeatletter
\DeclareRobustCommand{\cev}[1]{%
  \mathpalette\do@cev{#1}%
}
\newcommand{\do@cev}[2]{%
  \fix@cev{#1}{+}%
  \reflectbox{$\m@th#1\vec{\reflectbox{$\fix@cev{#1}{-}\m@th#1#2\fix@cev{#1}{+}$}}$}%
  \fix@cev{#1}{-}%
}
\newcommand{\fix@cev}[2]{%
  \ifx#1\displaystyle
    \mkern#22mu
  \else
    \ifx#1\textstyle
      \mkern#22mu
    \else
      \ifx#1\scriptstyle
        \mkern#22mu
      \else
        \mkern#22mu
      \fi
    \fi
  \fi
}

\makeatother

\begin{document}

\renewcommand{\thefootnote}{\fnsymbol{footnote}}
\footnotetext[1]{Department of Mathematics, University of Utah, 155 South 1400 East, Salt Lake City, UT 84109, USA. \EMAIL{janjigia@math.utah.edu}, \url{http://www.math.utah.edu/\textasciitilde janjigia}}

\section{Introduction and results}
Brownian directed percolation is a model of directed last passage percolation in a white noise environment on $\bbR \times \bbZ_+$. Given a family of i.i.d. standard two-sided Brownian motions $\{B_i\}_{i=0}^\infty$, we take the convention that $B_k(s,t) = B_k(t) - B_k(s)$ and define last passage times from $(u,k)$ to $(t,n)$, with $u \leq t \in \bbR$ and $k \leq n \in \bbZ_+$, by
\begin{align*}
L_{k,n}(u,t) &:= \displaystyle \sup_{u = s_{k-1} < s_k < \dots < s_{n-1} < s_n = t}\left\{ \sum_{j=k}^n B_{j}(s_{j-1},s_j) \right\}.
\end{align*}
We will distinguish the last passage times from $(0,1)$ to $(t,n)$ with the notation $L_n(t) = L_{1,n}(0,t):$
\begin{align}
L_n(t) &= \max_{0 < s_1 < \dots < s_{n-1} < t }\left\{B_1(0,s_1) + \dots + B_n(s_{n-1},t)\right\} \label{eq:BDPdef}.
\end{align}

Brownian directed percolation was originally introduced in the queueing literature in \cite{GW91}. It is rich with connections to other families of models including directed polymers and last passage percolation \cite{BS05,BM05,Ib11,OY01}, queueing theory \cite{GW91,OY01} and random matrix theory \cite{Bar01,GTW01,OY02}. The random matrix connection is of particular interest: a distributional equivalence between the last passage time $L_n(t)$ and the largest eigenvalue of an appropriate GUE matrix was discovered independently by Baryshnikov  \cite[Theorem 0.7]{Bar01} and Gravner, Tracy, and Widom \cite{GTW01} and soon after extended by O'Connell and Yor \cite{OY02}.

As is typical in percolation models, some of the main questions of interest concern the behavior of the passage times $L_n(nt)$ as $n\to\infty$. Superadditivity implies the existence of an almost sure law of large numbers limit $L_n(nt)/n \to c(t)$ as $n \to \infty$, for each $t>0$. Scale invariance inherited from Brownian motion, $L_n(t) \stackrel{d}{=} \sqrt{t} L_n(1)$, then shows that this limit must be of the form $c(t) = c \sqrt{t}$. One can either appeal to classical results in random matrix theory or arguments similar to those later in this paper to determine that $c=2$. The uniform version of this law of large numbers is due to Hambly, Martin, and O'Connell \cite{HMO02}, Theorem 8]. Through the connection to random matrices, the fluctuations around this limit are given by the Tracy-Widom GUE distribution \cite{TW94}. 

Large deviation principles are known at both rates $n$ and $n^2$. These results depend on the large deviation principle for the empirical distribution of a random matrix drawn from the Gaussian Unitary Ensemble, which is due to Ben Arous and Guionnet \cite[Theorem 1.3]{Ben-Gui-97}. The rate $n^2$ large deviation principle corresponds to the lower tail and was derived first, using the fact that a lower tail large deviation of the largest eigenvalue imposes a constraint on all of the eigenvalues. The rate $n$ upper tail large deviation rate function can be derived as in the computation of the corresponding rate function for the largest eigenvalue of a random matrix drawn from the Gaussian Orthogonal Ensemble in \cite[Theorem 6.2]{BDG01}. The precise expression here can be found in the lecture notes of Ledoux on deviation inequalities for largest eigenvalues \cite[(1.25)]{Ledoux}.
\begin{theorem} \label{thm:BDPRTRF}
For any $r \geq 0$, 
\begin{align*}
\lim_{n \to \infty} - n^{-1} \log \bbP\left( n^{-1} L_n( n)  \geq 2(1 + r)  \right) &= 4 \int_0^r \sqrt{x(x+2)}dx  :=J_{GUE}(r).
\end{align*}
\end{theorem}
In this paper, we provide a simple alternative proof of Theorem \ref{thm:BDPRTRF} using ideas which have previously been used to derive large deviation principles for the free energy of certain solvable positive and zero temperature directed polymer models in \cite{Cie-Geo-18-,Emr-Jan-17,GS13,Jan15,Sep98c,Sep98b}. The proof presented below shows that the limit in the statement of the theorem exists by subadditivity argument, from which we immediately derive the following corollary, which is also observed in \cite[(2.6)]{Ledoux}:
\begin{corollary} \label{cor:BDPfinitenbound}
For any $n$ and $r \geq 0$,
\begin{align*}
\bbP\left(n^{-1}L_n(n) \geq 2(1 + r) \right) \leq e^{-n J_{GUE}(r)}.
\end{align*}
\end{corollary}

The key input needed for the approach taken in this paper is an analogue of Burke's theorem from queueing theory, which allows us to define a version of the percolation model in which appropriate differences of passage times are stationary.

\subsection{The Burke property and the increment stationary model}
Following the notation in \cite{OY01},  for each $\mu > 0$, we define
\begin{align*}
q_1^\mu(t) = \sup_{-\infty < s \leq t}\left\{B_0(s,t) + B_1(s,t) - \mu(t-s)\right\}, \qquad
d_1^\mu(s,t) = B_0(s,t) + q_1^\mu(s) - q_1^\mu(t),
\end{align*}
and recursively for $k \geq 2$,
\begin{align*}
q_k^{\mu}(t) = \sup_{-\infty < s \leq t}\left\{d_{k-1}(s,t) + B_k(s,t) - \mu(t-s)\right\}, \qquad d_k^\mu(s,t) = d_{k-1}^\mu(s,t) + q_k^{\mu}(s) - q_k^{\mu}(t).
\end{align*}
These processes arise naturally in a heavy traffic limit in queueing theory. In that context, the increments of the Brownian motion $B_0$ represent the inter-arrivals process at the first queueing station, the increments of $B_k(t) - \mu t$  are the inter-service times at station $k$, the $d_k$ represents the inter-departure times from the $k^{th}$ station, and $q_k$ is the queue length process at station $k$. These definitions follow naturally from applying Donsker's theorem to the corresponding definitions in the classical M/M/1 queue. As is shown in \cite[Theorem 2, Section 4]{OY01} (see also \cite{Har-Wil-92} for more general results), applying Donsker's theorem to Burke's theorem for the M/M/1 queue leads to the following Brownian analogue of Burke's theorem:
\begin{theorem}\label{thm:BDPBurke}
For each $t \geq 0$, the family $\{q_k^{\mu}(t)\}_{k=1}^\infty$ consists of i.i.d.\ Exponential random variables with mean $\mu^{-1}$.
\end{theorem}

To build the increment stationary model, define a family of last passage times for $n \in \bbN$ and $t \in \bbR$ by
\begin{align} 
L_n^\mu(t) &=\displaystyle \sup_{- \infty < s_0< s_1  < \dots< s_{n-1} < s_n = t}\left\{\mu s_0 - B_0(s_0) + \sum_{j=1}^n B_j(s_{j-1},s_j)\right\} \label{eq:statdef} \\
&= \sup_{-\infty < s_0 < t}\left\{\mu s_0 - B_0(s_0) + L_{1,n}(s_0,t)\right\}. \notag
\end{align}
With these definitions, an induction argument shows that
\begin{align} \label{eq:BDPstatdecomp}
\sum_{k=1}^n q_k^{\mu}(t) &= B_0(t) - \mu t + L_n^\mu(t).
\end{align}
In particular, $\sum_{k=1}^n q_k^{\mu}(0) =  L_n^\mu(0).$ We think of paths in this extended directed percolation model as being indexed by the points where they exit the lines $\{0,\dots,n\}$. By grouping paths into those that exit line $0$ before time $0$ and those that exit after, we obtain
\begin{align}
L_n^\mu(t) &= \max_{0 \leq s_0 \leq t}\left\{\mu s_0 - B_0(s_0) + L_{1,n}(s_0,t)\right\} \vee \max_{1 \leq j \leq n}\left\{L_j^\mu(0) + L_{j,n}(0,t) \right\}. \label{eq:BDPstatcoupling}
\end{align}
The decomposition in (\ref{eq:BDPstatcoupling}) can be viewed as describing a `stationary' point-to-point polymer on $\bbR_+ \times \bbZ_+$ with i.i.d.~Exponential boundary conditions $\{L_{n+1}^\mu(0) - L_n^\mu(0)\}_{n \in \bbN}$ on the vertical axis and drifted Brownian boundary conditions $\{\mu t - B_0(t)\}_{t\geq 0}$ on the horizontal axis. The model is stationary in the sense that $\{L_{n+1}^\mu(t) - L_{n}^\mu(t)\}_{n\in \bbZ_+} =\{q_{n+1}^{\mu}(t) : n \in \bbZ_+\}$ is an i.i.d.\ Exponential family for each $t > 0$. We will combine the queueing picture with this finite $n$ variational problem in order to obtain a variational problem for the Lyapunov exponents in this model which will allow us to prove Theorem \ref{thm:BDPRTRF}. It is convenient to write the equality in \eqref{eq:BDPstatcoupling} in a way that separates the terms $\sum_{k=1}^n q_k^{\mu}(t)$ and $ B_0(t) - \mu t $, which are not independent:
\begin{align} \label{eq:BDPcoupling}
\sum_{k=1}^n q_k^{\mu}(t) &= \max_{0 \leq s_0 \leq t}\left\{\mu (s_0 - t) + B_0(t) - B_0(s_0) + L_{1,n}(s_0,t)\right\}\\
&\vee \max_{1 \leq j \leq n}\left\{B_0(t) - \mu t + \sum_{k=1}^j q_k^{\mu}(0) + L_{j,n}(0,t) \right\}. \notag
\end{align}
The key point in this decomposition is that for each $s_0 > 0$, the random variables $ B_0(t) - B_0(s_0)$ and $L_{1,n}(s_0,t)$ are independent and for each $j \in \{1, \dots, n\}$, the random variables $B_0(t)$, $\sum_{k=1}^j q_k^{\mu}(0)$, and $ L_{j,n}(0,t)$ are mutually independent. This independence can be seen by recalling that the Brownian motions $\{B_i\}_{i=0}^\infty$ are independent and observing that $\sigma(B_i(s) : s \leq 0, i \in \bbZ_+)$ and $\sigma(B_i(s) : s \geq 0, i \in \bbZ_+)$ are independent. This decomposition will lead to a variational problem which can be used to prove Theorem \ref{thm:BDPRTRF}. Once we have proven Theorem \ref{thm:BDPRTRF}, we can bootstrap that result and the decomposition in (\ref{eq:BDPstatcoupling}) to compute the corresponding positive moment Lyapunov exponents for the stationary model.
\begin{theorem} \label{thm:statLyapunov}
For each $\mu,s,t>0$ and $\lambda \geq 0$,
\begin{align*}
&\lim_{n\to\infty} \frac{1}{n}\log \bbE\left[e^{\lambda L_{\lf ns \rf}^\mu(nt)}\right] \\
&\qquad\qquad= \begin{cases} 
\left\{t\left(\frac{\lambda^2}{2} + \mu \lambda \right) + s \log \frac{\mu + \lambda}{\mu}\right\} \vee\left\{t\left(-\frac{\lambda^2}{2} + \mu \lambda \right) + s \log \frac{\mu}{\mu-\lambda}\right\} & \lambda < \mu \\
\infty & \lambda \geq \mu.
\end{cases}
\end{align*}
\end{theorem}
Note that although the previous expression has four parameters in it, only two (one of $s$ or $t$ and one of $\lambda$ or $\mu$) are really essential to the statement.  The expression above can be simplified by appealing to the equality in law $\lambda L_n^{\mu}(t) \stackrel{d}{=} L_n^{\mu/\lambda}(\lambda^2 t)$, which is valid for $\mu,\lambda>0$, $t \in \bbR$, and $n \in \bbN$. This follows immediately from the definition of $L_n(t)$ in \eqref{eq:statdef} and Brownian scaling. The form of the theorem given here with the extra parameters included is easier to work with in the proof.

One of the main goals of this paper is to keep the proofs short and non-technical. For this reason, we stop at computing the Lyapunov exponents for the stationary model and do not prove the corresponding upper tail rate function limit. The essential technical difficulty is that we no longer have access to subadditivity in the stationary model, which would have given us {\em a priori} existence, finiteness, and convexity of the upper tail rate function. If we knew those properties, then Theorem \ref{thm:statLyapunov} would give the rate function by taking a Legendre transform. Without this, the rate function can be computed following steps similar to the proof of Theorem \ref{thm:statLyapunov}, but working instead with $\limsup$ and $\liminf$ upper tail rate functions directly. The analysis and calculus needed to derive and solve the resulting variational problem become more involved than in the proof of Theorem \ref{thm:statLyapunov}. See \cite{GS13,Jan15} and in particular the proof of \cite[Theorem 2.14]{GS13} for a similar argument.

\section{Proof of Theorem \ref{thm:BDPRTRF} and Corollary \ref{cor:BDPfinitenbound}}
To start, we show existence and regularity of the Lypaunov exponents and upper tail rate functions in the point-to-point model. These are essentially immediate consequences of the superadditivity of the passage times.
\begin{proposition}\label{prop:BDPLJexist}
For any $s,t,\lambda>0$ and $r \in \bbR$, the limits
\begin{align*}
\Lambda_{s,t}(\lambda) := \lim_{n \to \infty} \frac{1}{n} \log \bbE\left[e^{\lambda L_{\lf ns \rf}(nt)}\right]  , \qquad J_{s,t}(r) := \lim_{n\to\infty} -\frac{1}{n} \log \bbP\left( L_{\lf ns \rf}(nt) \geq nr \right)
\end{align*}
exist and are finite. Moreover $J_{s,t}(r) \geq 0$. For each $\lambda > 0$, the map $(s,t) \mapsto \Lambda_{s,t}(\lambda)$ for $(s,t) \in (0,\infty)^2$ is positively homogeneous of degree one, superadditive, concave, and continuous. For $(s,t,r) \in (0,\infty)^2 \times \bbR$, the map $(s,t,r) \mapsto J_{s,t}(r)$ is positively homogeneous of degree one, subadditive, convex, and continuous. For each $(s,t)$, $J_{s,t}(r) = 0$ for $r \leq 2 \sqrt{st}$ and $r \mapsto J_{s,t}(r)$ is non-decreasing.
\end{proposition}

\begin{proof}
Note that the pre-limit expression in the definition of $J_{s,t}(r)$ is non-negative. For all of the conclusions except finiteness of $\Lambda_{s,t}(\lambda)$ and the last two properties of $J_{s,t}(r)$, it then suffices to show that the maps
\begin{align*}
(s,t) \mapsto \log \E\left[e^{\lambda L_{\lf s \rf}(t)}\right], \qquad (s,t,r) \mapsto - \log \bbP\left(L_{\lf s \rf}(t) \geq r\right)
\end{align*} 
are superadditive on $[1,\infty) \times (0,\infty)$ and subadditive on $[1,\infty)\times(0,\infty)\times \bbR$ respectively. See \cite[Theorem 16.2.9]{Kuczma} and the comment following the proof. Note that a subadditive function which is positively homogeneous of degree one is convex. Take $s_1,s_2 \geq 1$, $t_1,t_2 > 0$ and $r_1, r_2 \in \bbR$.  We have the inequality
\begin{align*}
L_{\lf (s_1 + s_2) \rf}(t_1 + t_2) &\geq L_{\lf s_1 \rf}(t_1) + L_{\lf s_1 \rf, \lf (s_1 + s_2)\rf}(t_1, (t_1 + t_2))
\end{align*}
where the last two terms are independent. Using translation invariance, independence, and monotonicity of $L_n(t)$ in $n$, we have
\begin{align*}
\E\left[e^{\lambda L_{\lf s_1 + s_2 \rf}(t_1 + t_2)}\right] &\geq \bbE\left[e^{\lambda L_{\lf s_1 \rf}(t_1)}\right]  \bbE\left[e^{\lambda L_{\lf s_2 \rf}(t_2)}\right], \\
\bbP\left(L_{\lf s_1 + s_2 \rf}(t_1 + t_2) \geq r_1 + r_2\right)&\geq \bbP\left( L_{\lf s_1 \rf}(t_1) \geq r_1\right) \bbP\left( L_{\lf s_2 \rf}(t_2) \geq r_2\right).
\end{align*}
Finiteness of $\Lambda_{s,t}(\lambda)$ for all $\lambda > 0$ follows from $L_n(t) \leq \sum_{i=0}^n 2 \max_{0 \leq r \leq t}|B_i(r)|$. The properties of $J_{s,t}(r)$ follow from the almost sure limit $L_n(nt)/n \to 2\sqrt{t}$, continuity, and the fact that the pre-limit expression is non-decreasing in $r$.
\end{proof}

\begin{remark} \label{rem:BDPfinitenbound}
Subadditivity shows that $J_{1,1}(r) = \inf_n  -n^{-1}\log \bbP\left( L_n(n) \geq nr \right)$. As a consequence, for any $n$, we have $\bbP\left( L_{n}(n) \geq nr \right)\leq \text{exp}\left\{-n J_{1,1}(r)\right\}$.
\end{remark}

The next result shows that the decomposition in (\ref{eq:BDPcoupling}) implies that $\Lambda_{s,t}(\lambda)$ is the solution to an invertible variational problem. This type of decomposition and versions of the argument that follows are the key steps in the papers \cite{Cie-Geo-18-,Emr-Jan-17,GS13,Jan15}.

\begin{lemma}\label{lem:BDPvarprob}
For each $s,t > 0$ and $\lambda \in (0,\mu)$,
\begin{align*}
s \log \frac{\mu}{\mu-\lambda} &= \sup_{0 \leq  r < t}\left\{(t-r) \left( \frac{\lambda^2}{2} -\mu \lambda\right) + \Lambda_{s,t-r}(\lambda) \right\}\\ 
&\vee \sup_{0 \leq u < s}\left\{t\left(\frac{1}{2}\lambda^2  - \mu \lambda\right) + u \log \frac{\mu}{\mu-\lambda} + \Lambda_{s-u,t}(\lambda)\right\}.
\end{align*}
\end{lemma}
\begin{proof}
We begin with the coupling (\ref{eq:BDPcoupling}). It follows that for any $r \in [0,t)$, $u \in [0,s)$ and $n$ large enough,
\begin{align*}
\bbE\left[e^{\lambda \sum_{k=1}^{\lf ns \rf} q_k^{\mu}(nt)}\right] &\geq \bbE\left[ e^{\lambda\left( \mu (nr - nt) + B_0(nt) - B_0(nr) + L_{1,n}(nr,nt)\right)}\right] \\
 &\vee \bbE\left[e^{\lambda\left(B_0(t) - \mu t + \sum_{k=1}^{\lf nu \rf} q_k^{\mu}(0) + L_{\lf nu \rf,n}(0,t)\right)}\right].
 \end{align*}
The random variables $B_0(nt) - B_0(nr)$ and $L_{1,n}(nr,nt)$ are independent because $B_0(\cdot)$ is independent of $\{B_j(\cdot)\}_{j=1}^\infty$. The random variables  $\sum_{k=1}^{\lf nu \rf} q_k^{\mu}(0)$, $L_{\lf nu \rf,n}(0,t)$, and $B_0(t)$ are independent because $\sum_{k=1}^{\lf nu \rf} q_k^{\mu}(0)$ is measurable with respect to $\sigma(B_j(t) : t \leq 0, j \in \bbZ_+)$, $L_{\lf nu \rf,n}(0,t)$ is measurable with respect to $\sigma(B_j(t) : t \geq 0, j \in \bbN)$, and $B_0(t)$ is measurable with respect to $\sigma(B_0(t) : t >0)$. Taking logs, dividing by $n$ and sending $n \to \infty$, and optimizing over $u$ and $r$, we immediately obtain $\geq$ in the statement of the theorem.

Let $\{r_i\}_{i=1}^M$ and $\{u_i\}_{i=1}^M$ be partitions of $[0,t]$ and $[0,s]$ into equally sized subintervals of length $t/M$ and $s/M$ respectively. Notice that
\begin{align*}
&\max_{0 \leq r \leq t}\left\{n\mu (r - t) + B_0(nt) - B_0(nr) + L_{1,\lf ns\rf}(nr,nt)\right\} \\
= &\max_{2 \leq i \leq M} \max_{r \in [r_{i-1},r_i]} \left\{n\mu (r - t) + B_0(nt) - B_0(ns) + L_{1,\lf ns \rf}(nr,nt)\right\} \\
\leq &\max_{2 \leq i \leq M} \left\{n \mu (r_i - t) +  B_0(nt) - B_0(n r_{i}) +  \max_{r \in [r_{i-1},r_i]} \left\{B_0(n r_i) - B_0(nr)\right\}  + L_{1, \lf ns \rf}(r_{i-1},t)\right\}.
\end{align*}
Similarly, we have
\begin{align*}
&\max_{1 \leq j \leq \lf ns \rf}\left\{B_0(n t) - n \mu t + \sum_{k=1}^j q_k^{\mu}(0) + L_{j,\lf ns \rf}(0,nt) \right\} \\
\leq &\max_{2 \leq i \leq M} \left\{B_0(nt) - n \mu t + \sum_{k=1}^{\lf n u_{i} \rf} q_k^{\mu}(0) + L_{\lf n u_{i-1}\rf,\lf n s \rf}(0,nt)\right\}.
\end{align*}
It follows from these inequalities and independence that 
\begin{align*}
&\bbE\left[e^{\lambda \sum_{j=1}^{\lf ns \rf}q_k^\mu(nt)}\right] \leq\\
&\sum_{i=2}^M e^{n \mu (r_i - t)}\bbE\left[e^{\lambda (B_0(nt) - B_0(n r_{i}))}\right]\bbE\left[e^{\lambda  \max_{r \in [r_{i-1},r_i]} \left\{B_0(n r_i) - B_0(nr)\right\}}\right]\bbE\left[e^{\lambda L_{1, \lf ns \rf}(r_{i-1},t)}\right] \\
&\qquad+ \bbE\left[e^{\lambda(B_0(nt) - n \mu t )}\right]\bbE\left[e^{\lambda  \sum_{k=1}^{\lf n u_{i} \rf} q_k^{\mu}(0)  }\right]\bbE\left[e^{\lambda L_{\lf n u_{i-1}\rf,\lf n s \rf}(0,nt)}\right]
\end{align*}
By the reflection principle and the assumption that $r_i - r_{i-1} = \frac{t}{M}$, we have
\begin{align*}
&\bbE\left[e^{\lambda  \max_{r \in [r_{i-1},r_i]} B_0(n r_i) - B_0(nr)} \right] = \bbE\left[e^{\lambda \sqrt{n} |B_0(\frac{t}{M})|}\right] \\
&\qquad \leq \bbE\left[e^{\lambda \sqrt{n} B_0(\frac{t}{M})}\right] + \bbE\left[e^{-\lambda \sqrt{n} B_0(\frac{t}{M})}\right].
\end{align*} 
Take logs, divide by $n$ and send $n \to \infty$ to obtain
\begin{align*}
s \log \frac{\mu}{\mu-\lambda} &\leq \max_{2 \leq i \leq M} \left\{\mu \lambda(r_i-t) + (t-r_i) \frac{\lambda^2}{2} + \frac{\lambda^2 t}{2M} + \Lambda_{s,t- r_{i-1}}(\lambda) \right\}\\ 
&\vee \max_{2 \leq i \leq M}\left\{\frac{1}{2}\lambda^2 t - \mu \lambda t + u_i \log \frac{\mu}{\mu-\lambda} + \Lambda_{s-u_{i-1},t}(\lambda)\right\} \\
&\leq \left(\sup_{0 \leq  r < t}\left\{\mu \lambda(r-t) + (t-r) \frac{\lambda^2}{2} + \Lambda_{s,t-r}(\lambda) \right\} + \frac{\lambda^2 t}{2M} + \frac{\mu \lambda t}{ M}  \right) \\
&\vee\left( \sup_{0 \leq u < s}\left\{\frac{1}{2}\lambda^2 t - \mu \lambda t + u \log \frac{\mu}{\mu-\lambda} + \Lambda_{s-u,t}(\lambda)\right\} + \frac{s}{M} \log \frac{\mu}{\mu-\lambda}\right).
\end{align*}
Sending $M \to \infty$ completes the proof.
\end{proof}
Variational problems of the type in Lemma \ref{lem:BDPvarprob} appear for the Lyapunov exponents and free energies (resp. time constants) of directed polymers (resp. percolation models) which have associated stationary models that satisfy appropriate analogues of the Burke property. Up to a change of variables, a deformation of the region on which the maximization takes place, and homogeneity of $\Lambda_{s,t}(\lambda)$ in $(s,t)$, this variational expression gives a Legendre-Fenchel duality between directions  $(s,t)$ and values of $\mu > \lambda$. See for example \cite[Section 5]{Emr-16} for this point of view. Alternatively, this variational problem can be solved directly with some easy calculus. This is done in some generality in \cite[Proposition 3.10]{Jan15}, so we appeal to that result here.
\begin{corollary} \label{cor:BDPLyapunov}
For any $s,t,\lambda > 0$, 
\begin{align*}
\Lambda_{s,t}(\lambda) &= \min_{\mu > \lambda }\left\{t \left(\lambda \mu - \frac{1}{2}\lambda^2\right) + s \log \frac{\mu}{\mu - \lambda }\right\} = \min_{z > 0}\left\{t \left(\frac{1}{2}\lambda^2 + z \lambda \right) + s \log \frac{z + \lambda}{z}\right\} \\
&= \frac{1}{2}\lambda\sqrt{4st + (t\lambda)^2} + s \log \left(\frac{2s + t \lambda^2 + \lambda \sqrt{4st + (t\lambda)^2}}{2s}\right) = \int_{0}^\lambda \sqrt{4st + (tx)^2} dx.
\end{align*}
\end{corollary}
\begin{proof}
The first equality follows from Lemma \ref{lem:BDPvarprob} and \cite[Proposition 3.10]{Jan15} with  $I = \{\mu > \lambda\}$, $h(\mu) = -\frac{\lambda^2}{2} + \lambda \mu$, and $g(\mu) = \log \frac{\mu}{\mu - \lambda}$. The second equality is the change of variables $z = \mu - \lambda$. The third and fourth equalities follow from calculus.
\end{proof}

The next result is the analogue of Varadhan's lemma for upper tail rate functions.
\begin{lemma} \label{lem:BDPvaradhan}
For each $s,t > 0$,
\begin{align*}
\sup_{r \in \bbR}\{\lambda r - J_{s,t}(r)\} &=\begin{cases}
\infty & \lambda < 0 \\
 \Lambda_{s,t}(\lambda)  & \lambda \geq 0.
\end{cases}
\end{align*}
\end{lemma}
\begin{proof}
The result for $\lambda \leq 0$ follows from the observations $J_{s,t}(r) \geq 0$ for all $r$, $J_{s,t}(r) = 0$ for $r \leq 2 \sqrt{st}$ and $r \mapsto J_{s,t}(r)$ is non-decreasing. Take $\lambda,K > 0$, and let $\{m_i\}_{i=1}^M$ be a uniform partition of $[0,K]$. The exponential Markov inequality yields for each $r > 0$
\begin{align}\label{eq:BDPexpmarkov}
\lambda r - J_{s,t}(r) \leq \Lambda_{s,t}(\lambda).
\end{align}
Optimizing over $r$ gives $\leq$ in the statement of the lemma. For the reverse, notice that
\begin{align*}
\bbE\left[e^{\lambda L_{\lf ns \rf}(nt)}\right] &= \sum_{i=1}^M \bbE\left[e^{\lambda L_{\lf ns \rf}(nt)}1_{\{L_{\lf ns \rf}(nt) \in [m_{i-1}, m_i)\}}\right] + \bbE\left[e^{\lambda L_{\lf ns \rf}(nt)}1_{\{L_{\lf ns \rf}(nt) \geq K\}}\right]\\
&\leq \sum_{i=1}^M e^{\lambda m_i}\bbP\left(L_{\lf ns \rf}(nt) \geq m_{i-1}\right) + \bbE\left[e^{\lambda L_{\lf ns \rf}(nt)}1_{\{L_{\lf ns \rf}(nt) \geq K\}}\right] \\
&\leq \sum_{i=1}^M e^{\lambda m_i}\bbP\left(L_{\lf ns \rf}(nt) \geq m_{i-1}\right) + \bbE\left[e^{2\lambda L_{\lf ns \rf}(nt)}\right]^{\frac{1}{2}}\bbP\left(L_{\lf ns \rf}(nt) \geq K\right)^\frac{1}{2}.
\end{align*}
Take logs, divide by $n$ and send $n \to \infty$ to obtain
\begin{align*}
\Lambda_{s,t}(\lambda) &\leq \max_{i \leq M}\left\{\lambda m_i - J_{s,t}(m_{i-1})\right\} \vee\left\{\frac{1}{2}\Lambda_{s,t}(2\lambda) - \frac{1}{2}J_{s,t}(K)\right\} \\
&\leq \left(\sup_{r \in \bbR}\left\{\lambda r - J_{s,t}(r)\right\} + \frac{\lambda}{M}\right) \vee \left\{\frac{1}{2}\Lambda_{s,t}(2\lambda) - \frac{1}{2}J_{s,t}(K)\right\}.
\end{align*}
Equation (\ref{eq:BDPexpmarkov}) shows that $J_{s,t}(K) \to \infty$ as $K \to \infty$. Sending $M,K \to \infty$ completes the proof.
\end{proof}
\begin{corollary}
For $s,t > 0$ and $r \geq 2 \sqrt{st}$,
\begin{align*}
J_{s,t}(r) &= \sup_{\lambda,z > 0}\left\{\lambda r - t \left(\frac{1}{2}\lambda^2 + z \lambda \right) - s \log \frac{z + \lambda}{z}\right\} = \frac{r\sqrt{r^2-4st}}{2t} + s \log \left(\frac{r-\sqrt{r^2-4st}}{r+\sqrt{r^2-4st}}\right).
\end{align*}
\end{corollary}
\begin{proof}
The first equality follows from Lemma \ref{lem:BDPvaradhan}, Corollary \ref{cor:BDPLyapunov}, and the Fenchel-Moreau theorem \cite[Theorem 12.2]{Rock}. The second equality can be obtained with calculus.
\end{proof}
\begin{remark}
Differentiating the expression in the previous result gives
\begin{align*}
J_{s,t}(r) &= 1_{\{r \geq 2\sqrt{st}\}} \int_0^{r - 2 \sqrt{st}} t^{-1}\sqrt{x(x+4\sqrt{st})}dx .
\end{align*} 
Setting $s=t=1$ and changing variables gives the expression in Theorem \ref{thm:BDPRTRF}. Combining this result with Remark \ref{rem:BDPfinitenbound} gives Corollary \ref{cor:BDPfinitenbound}.
\begin{align*}
\end{align*}
\end{remark}

\section{Proof of Theorem \ref{thm:statLyapunov}}
Having computed $\Lambda_{s,t}(\lambda)$, (\ref{eq:BDPstatcoupling}) now leads to a variational problem for the Lyapunov exponents in the stationary model for each $\mu > \lambda$. Using Corollary \ref{cor:BDPLyapunov} we may extend $\Lambda_{s,t}(\lambda)$ continuously to $\Lambda_{0,t}(\lambda) = \frac{\lambda^2t}{2}$ and $\Lambda_{s,0}(\lambda) = 0$.
\begin{lemma} \label{lem:BDPstatLyapunov1}
For each $\mu,s,t>0$ and $\lambda \in (0,\mu)$,
\begin{align*}
&\lim_{n \to \infty} \frac{1}{n}\log \bbE\left[e^{\lambda L_{\lf ns \rf}^\mu(nt)}\right]\\
&\qquad\qquad= \sup_{0 \leq r \leq t}\left\{r\left(\lambda \mu  + \frac{\lambda^2}{2}\right) + \Lambda_{s,t-r}(\lambda) \right\}\vee \sup_{0\leq u \leq s}\left\{u \log \frac{\mu}{\mu-\lambda} + \Lambda_{s-u,t}(\lambda)\right\} \\
&\qquad\qquad=  \left\{t\left(\frac{\lambda^2}{2} + \mu \lambda \right) + s \log \frac{\mu + \lambda}{\mu}\right\} \vee\left\{t\left(-\frac{\lambda^2}{2} + \mu \lambda \right) + s \log \frac{\mu}{\mu-\lambda}\right\}.
\end{align*}
\end{lemma}
\begin{proof}
The proof of the first equality is essentially the same as in the proof of Lemma \ref{lem:BDPvarprob}, except that one must work with $\liminf$ and $\limsup$. For example, for any $r \in [0,t), u \in [0,s)$, and $n$ sufficiently large, we have
\begin{align*}
&\bbE\left[e^{\lambda L_{\lf ns \rf}^\mu(nt)}\right] \\
&\qquad\qquad\geq \bbE\left[e^{\lambda (\mu n r - B_0(r))}\right]\bbE\left[e^{\lambda L_{1, \lf ns \rf}(r,nt)}\right] \vee \bbE\left[e^{\lambda \sum_{i=1}^{\lf nu \rf}q_k^\mu(0)}\right]\bbE\left[e^{\lambda L_{\lf nu \rf, \lf ns \rf}(0,nt)}\right].
\end{align*}
Take logs, divide by $n$, take $\liminf$, and optimize to obtain
\begin{align*}
&\liminf_{n \to \infty} \frac{1}{n} \log \bbE\left[e^{\lambda L_{\lf ns \rf}^\mu(nt)}\right] \\
&\qquad\qquad\geq \sup_{0 \leq r \leq t}\left\{r\left(\lambda \mu + \frac{\lambda^2}{2}\right) + \Lambda_{s,t-r}(\lambda) \right\}\vee \sup_{0\leq u \leq s}\left\{u \log \frac{\mu}{\mu-\lambda} + \Lambda_{s-u,t}(\lambda)\right\}.
\end{align*}
We omit the reverse inequality which similarly follows from the type of arguments in the proof of Lemma \ref{lem:BDPvarprob}. For the second equality, it is convenient to substitute $r \mapsto t -r$ and $u \mapsto s - u$. Using the second variational expression for $\Lambda_{s,r}(\lambda)$ from Corollary \ref{cor:BDPLyapunov} and a minimax theorem (for example, see \cite[Appendix B.3]{Rassoul-AghaSeppalainen}), we obtain
\begin{align*}
 &\max_{0 \leq r \leq t}\left\{r \left(\frac{\lambda^2}{2} + \lambda \mu\right) + \Lambda_{s,t-r}(\lambda) \right\} = t\left(\frac{\lambda^2}{2} + \mu \lambda \right) +  \min_{z > 0} \max_{0 \leq r \leq t}\left\{r \lambda (z-\mu) + s \log \frac{z+\lambda}{z} \right\} \\
&\qquad\qquad\qquad\qquad\qquad=  \left\{t\left(\frac{\lambda^2}{2} + \mu \lambda \right) + s \log \frac{\mu + \lambda}{\mu}\right\} \wedge \min_{z \geq \mu}\left\{t \left(\frac{\lambda^2}{2} + z \lambda\right) + s \log \frac{z+\lambda}{z}\right\}.
\end{align*}
The second equality comes from dividing the minimum into the regions $z \leq \mu$ and $z > \mu$. A similar argument using the same variational expression and dividing into $z \leq \mu - \lambda$ and $z > \mu - \lambda$ shows that
\begin{align*}
\max_{0\leq u \leq s}\left\{u \log \frac{\mu}{\mu-\lambda} + \Lambda_{s-u,t}(\lambda)\right\} &= \left\{t\left(-\frac{\lambda^2}{2} + \mu \lambda \right) + s \log \frac{\mu}{\mu-\lambda}\right\}\\
&\wedge \min_{z \leq \mu - \lambda}\left\{t\left(\frac{\lambda^2}{2} + z \lambda\right) + s \log \frac{z+\lambda}{z} \right\}.
\end{align*}
In both of these expressions, the first term is feasible in the minimum over $z$ by taking $z = \mu$ or $z = \mu - \lambda$. To complete the proof, note that the function being minimized in the second variational expression $\Lambda_{s,r}(\lambda)$ in Corollary \ref{cor:BDPLyapunov}  is strictly convex and minimizers exist.
\end{proof}
We complete this section and the paper by addressing the exponents $\lambda \geq \mu$. This is an immediate corollary of the previous lemma.
\begin{corollary}
For each $\mu,s,t >0$ and $\lambda \geq \mu$,
\begin{align*}
\lim_{n \to \infty} \frac{1}{n} \log \bbE\left[e^{\lambda L_{\lf ns \rf}^\mu (nt)}\right] = \infty.
\end{align*}
\end{corollary}
\begin{proof}
The function $\lambda \mapsto \log \bbE\left[e^{\lambda L_{\lf ns \rf}^\mu (nt)}\right]$ is non-decreasing. It follows that for any $\lambda < \mu$,
\begin{align*}
\liminf_{n \to \infty} \frac{1}{n} \log \bbE\left[e^{\mu L_{\lf ns \rf}^\mu (nt)}\right] \geq t\left(\frac{\lambda^2}{2} + \lambda \mu\right) + s \log \frac{\mu}{\mu-\lambda}.
\end{align*}
Sending $\lambda \uparrow \mu$ gives the result.
\end{proof}

%%%%%%%%%%%%%%%%%%%%%%%%%%%%%%%%%%%%%%%%%%%%%%%%%%%%%%%%%%%%%%%%%%%
%%                                                               %%
%% Use the two commands below for producing your bibliography    %%
%% with bibtex, then comment again \thispagestyle{?}e commands and include the  %%
%% content of the .bbl file in this file below the commands.     %%
%%                                                               %%
%%%%%%%%%%%%%%%%%%%%%%%%%%%%%%%%%%%%%%%%%%%%%%%%%%%%%%%%%%%%%%%%%%%

%\bibliographystyle{amsplain}
%\bibliography{BDP_LDP}

% add below the content of your .bbl file produced by bibtex.

\providecommand{\bysame}{\leavevmode\hbox to3em{\hrulefill}\thinspace}
\providecommand{\MR}{\relax\ifhmode\unskip\space\fi MR }
% \MRhref is called by the amsart/book/proc definition of \MR.
\providecommand{\MRhref}[2]{%
  \href{http://www.ams.org/mathscinet-getitem?mr=#1}{#2}
}
\providecommand{\href}[2]{#2}

%%%%%%%%%%%%%%%%%%%%%%%%%%%%%%%%%%%%%%%%%%%%%%%%%%%%%%%%%%%%%%%%%%%
%%                                                               %%
%% You may add acknowledgments (optional).                       %%
%%                                                               %%
%%%%%%%%%%%%%%%%%%%%%%%%%%%%%%%%%%%%%%%%%%%%%%%%%%%%%%%%%%%%%%%%%%%

\ACKNO{The author is grateful to Firas Rassoul-Agha and an anonymous referee for helpful comments on previous versions of the manuscript.}

%%%%%%%%%%%%%%%%%%%%%%%%%%%%%%%%%%%%%%%%%%%%%%%%%%%%%%%%%%%%%%%%%%%
%%                                                               %%
%% You have reached the end of your document.                    %%
%%                                                               %%
%%%%%%%%%%%%%%%%%%%%%%%%%%%%%%%%%%%%%%%%%%%%%%%%%%%%%%%%%%%%%%%%%%%

\end{document}